\documentclass[journal, a4paper]{IEEEtran}

\usepackage[numbers,sort&compress,square]{natbib}   % used for the advanced citer
\usepackage{graphicx}
\usepackage{subfigure}
\usepackage{amsmath,amssymb}
\usepackage{amsfonts}
\usepackage{array}
\usepackage{amsthm}
\usepackage{mathrsfs}
\usepackage{amsmath}
\usepackage{setspace}
\usepackage{color}
\usepackage[ruled,linesnumbered]{algorithm2e}
\usepackage{epstopdf}
\usepackage{soul}
\usepackage{mathrsfs}

\newcommand{\beq}{\begin{equation}}
 \newcommand{\eeq}{\end{equation}}
  
 \newcommand{\bea}{\begin{eqnarray}}
 \newcommand{\eea}{\end{eqnarray}}

\newcommand{\argmin}{\operatornamewithlimits{arg\,min}}

\renewcommand{\IEEEQED}{\IEEEQEDclosed}
\newtheorem{theorem}{Theorem}

\theoremstyle{definition}
\newtheorem{define}[theorem]{Definition}

\usepackage{color}

\begin{document}

\title{Age of Information: Whittle Index for Scheduling Stochastic Arrivals}

\author{\IEEEauthorblockN{Yu-Pin Hsu\\}
\IEEEauthorblockA{Department of Communication Engineering\\ National Taipei University\\}
\IEEEauthorblockA{{yupinhsu@mail.ntpu.edu.tw}}
}

%\setenumerate{noitemsep}
%\setlist{nosep}

%\captionsetup{font={small,bf,stretch=1}}

\maketitle

\begin{abstract}

\textit{Age of information} is a new concept that characterizes the freshness of  information at end devices. This paper studies the age of information from a scheduling perspective. We consider a wireless broadcast network where a base-station updates many users on \textit{stochastic} information arrivals. Suppose that only one user can be updated  for each time. In this context, we aim at developing a transmission scheduling algorithm for  minimizing the long-run average age. To develop a low-complexity transmission scheduling algorithm,   we  apply the Whittle's framework for restless bandits. We successfully derive the Whittle index in a closed form  and  establish the \textit{indexability}. Based on the Whittle index, we propose a scheduling algorithm, while experimentally  showing that it closely approximates an age-optimal scheduling algorithm.
\end{abstract}

\section{Introduction}
In recent years,  there has been a dramatic proliferation of research on  an \textit{age of information}. The age of information is inspired by a variety of network applications requiring  \textit{timely} information to accomplish  some tasks. Examples  include  information updates for \textit{smart-phone  users}, e.g.,  traffic and transportation, as well as status updates for \textit{smart systems}, e.g., smart transportation systems and smart grid systems. 

On the one hand, a smart-phone user needs timely  traffic and transportation information for planning the best route. On the other hand, timely information about vehicles' positions and speeds is needed for planning a collision-free transportation system. In both cases, snapshots of the information are generated by their sources at some epochs  and sent to the end devices (e.g., smart-phone users and vehicles) in the form of packets over wired or wireless networks. Since the information at the end devices is expected to be as timely as possible, the age of information is therefore proposed to capture the \textit{freshness of the information at the end devices}; more precisely, it measures the elapsed time since the generation of the freshest packets.  
The goal is to develop networks supporting the age-sensitive applications.  It is interesting  that either throughput-optimal design or delay-optimal design might not result in the minimum age \cite{age:kaul}.

%Before going any further, we would like to characterize the age of information in two aspects. First, while packet delay usually refers to elapsed time from the generation to its delivery, the age includes not only the packet delay but also the amount of time from the delivery until the present,  because the age of information increases until the information is updated. Hence,  delay-optimal design cannot minimize the average age. Second, while traditional relays (i.e., intermediate nodes)  need to keep all unsent packets, relays for timely information might  store  the latest information and discard out-of-date packets. Due to the distinctions, the research on the age of information is unique. 

In this paper, we consider that a base-station (BS) updates many users over a wireless broadcast network, where  new information is randomly generated. We assume that the BS can update at most one user for each transmission opportunity. Under the transmission constraint, a transmission scheduling algorithm manages how the channel resources  are allocated for each time, depending on the packet arrivals and the ages of the information. The scheduling  design is a critical issue to provide good network performance. We hence design and analyze a scheduling algorithm to minimize the long-run average age.

The wireless broadcast network  is similar to the model in the earlier work \cite{hsuage} of the present author; however,  a low-complexity scheduling algorithm is unexplored. To fill this gap, this work investigates an age-optimal scheduling from the perspective of \textit{restless bandits} \cite{gittins2011multi}. Whittle \cite{whittle} considers a  \textit{relaxed restless multi-armed bandit problem} and  decouple the problem into many  sub-problems consisting of  a single bandit, while proposing an index policy and a concept of \textit{indexability}. The Whittle index policy is asymptotically optimal under certain conditions \cite{weber1990index}, and in practice performs strikingly well \cite{larranaga2015stochastic}. 
Note that each user in our problem can be viewed as a restless bandit; as such, we  apply the  Whittle's approach to develop a scheduling algorithm.

\subsection{Contributions}
We transform our problem into a relaxed restless multi-armed bandit problem and investigate the problem from the Whittle's perspective. However, in general,  a closed form of the Whittle index might be  unavailable. To tackle this issue, we formulate  each decoupled sub-problem  as a Markov decision process (MDP), with the purpose of minimizing an average cost. Since our MDP involves an \textit{average cost} optimization  over \textit{infinite horizon} with a \textit{countably infinite state space},  our problem is challenging to  analyze \cite{MDP:Bertsekas}. We prove that an optimal policy  of the MDP  is \textit{stationary} and \textit{deterministic}; in particular, it is a simple \textit{threshold} type.  We then derive an optimal threshold  by exploiting the threshold type along with a \textit{post-action age}. It  turns out that the post-action age simplifies the calculation of the average cost; as such, we  successfully obtain the Whittle index in a closed form and show the \textit{indexability}. 
Finally, we propose a Whittle index  scheduling algorithm and numerically validate its performance.

\subsection{Related works}
The \textit{age of information}  has attracted many interests from the research community, e.g.,  \cite{age:kaul, age:Costa,bedewy2017age} and see the survey \cite{kosta2017age}. Of the most relevant works on scheduling multiple users are \cite{age:he,joowireless,age:igor,yatesage}. The works \cite{age:he,joowireless} consider queues at a BS to store all \textit{out-of-date} packets, different from ours. The paper \cite{age:igor} considers a buffer to store the \textit{latest} information with periodic arrivals; whereas information updates in \cite{yatesage} can be generated at will. Our work contributes to the age of information by developing a low-complexity algorithm for scheduling \textit{stochastic} information arrivals.

\section{System overview} \label{section:system}

\subsection{Network model} \label{subsection:model}
\begin{figure}[!t]
\centering
\includegraphics[width=.3\textwidth]{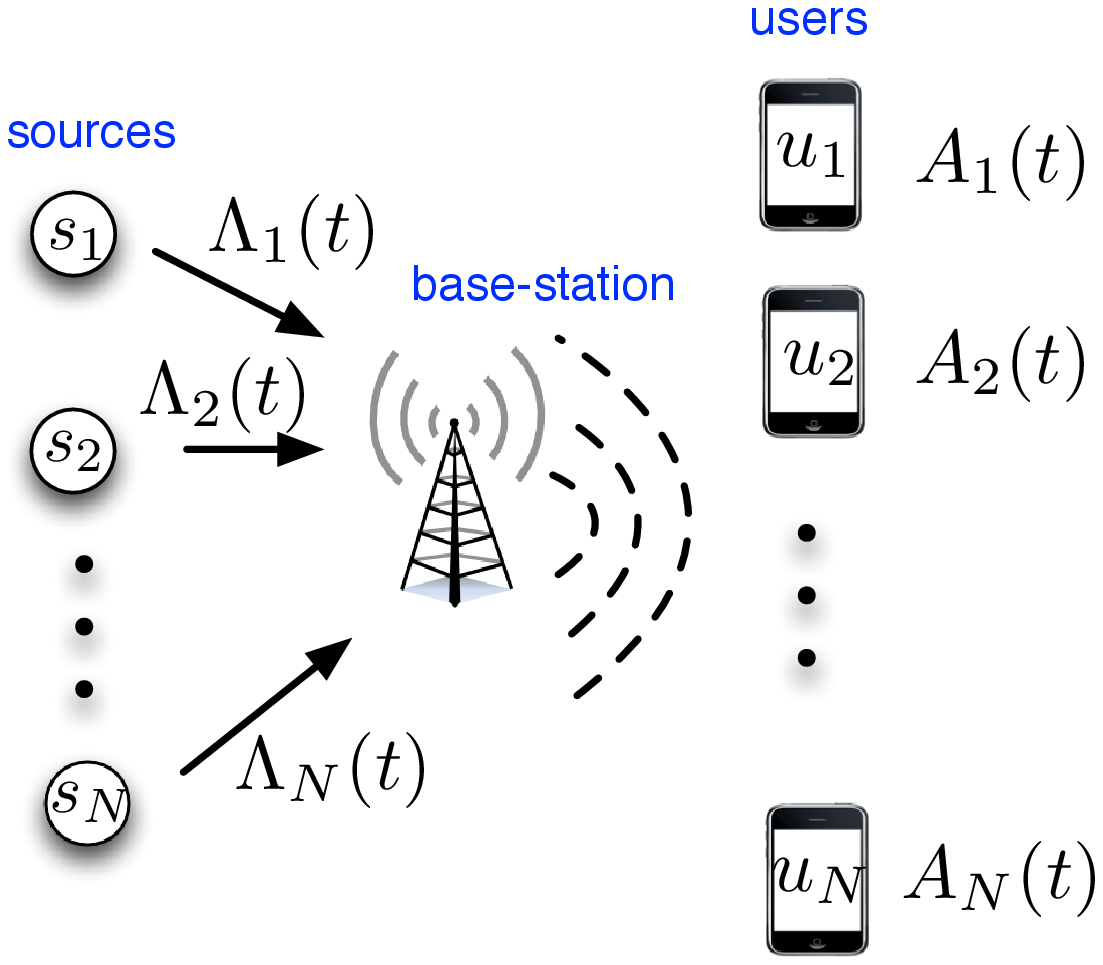}
\caption{A BS  updates $N$ users $u_1, \cdots, u_N$ on  information from sources $s_1, \cdots, s_N$, respectively.}
\label{fig:model}
\end{figure}

We consider a wireless broadcast network consisting of a base-station (BS) and $N$ wireless users $u_1, \cdots, u_N$  in Fig.  \ref{fig:model}. Each user $u_i$ is interested in a type of information generated by a source $s_i$, for $i=1, \cdots, N$, respectively. All information is sent in the form of packets by the BS over a \textit{noiseless} broadcast channel.

We consider a discrete-time system with slot \mbox{$t=0, 1,  \cdots$}. The packets from the sources (if any) arrive at the BS at the \textit{beginning} of each  slot.  The arrivals at the BS for different users are independent of each other, and also  independent and identically distributed (i.i.d.) over slots, governed by  a Bernoulli distribution.  Precisely, by $\Lambda_i(t)$ we indicate if a packet from source $s_i$ arrives at the BS in slot $t$, where $\Lambda_i(t)=1$ if there is a packet, with $P[\Lambda_i(t)=1]=p_i$; otherwise, $\Lambda_i(t)=0$.  

We assume that  the BS can send at most one packet during each  slot, i.e., the BS can update at most one user in each slot. Moreover, we focus on the setting that the BS does not buffer a packet if it is not transmitted in the arriving slot. The \textit{no-buffer network} is motivated by  \cite{hsuage}. 

By $D(t) \in \{0, 1, \cdots N\}$ we denote a decision of the BS in slot $t$, where $D(t)=0$ if no one will be updated in slot $t$ and $D(t)=i$ for $i=1, \cdots, N$ if  user $u_i$ is scheduled to be updated in slot $t$.   A \textit{scheduling algorithm} $\theta=\{D(0), D(1), \cdots\}$ specifies a  decision for each slot.  Next, we will define an \textit{age of information} as our design criterion. 

\subsection{Age of information model}
The age of information implies the freshness of the information \textit{at the users}. We initialize the ages of all arriving packets \textit{at the BS} to be zero. The age of information \textit{at a user}  becomes one on receiving a packet, due to one slot  of the transmission time. 
Let $X_i(t)$ be the age of information for user $u_i$ in slot $t$ \textit{before} the BS makes a  scheduling decision. Suppose that the age increases linearly with slots. Then, the dynamics of the age of information for user $u_i$ is
\begin{align*}
X_i(t+1)=\left\{
\begin{array}{ll}
1 & \text{if $\Lambda_i(t)=1$ and $D(t)=i$;}\\
X_i(t)+1 & \text{else,}
\end{array}
\right.
\end{align*}
where the age in the next slot is one if the user gets updated on the new information; otherwise, the  age increases by one.  
%The age of user $u_i$ in slot $t$ is $X_i(t)=t-r_i(t)$, where $r_i(t)$ is the slot that the latest packet received by user $u_i$ was generated.  Take Fig. \ref{fig:age-define} for example, where the arrivals in slots $t=2, 5$ for user $u_i$ are delivered in slots $t=3, 6$ and hence the  dots represent the ages of information for user $u_i$ for each slot. 
%
%\begin{figure}[!t]
%\centering
%\includegraphics[width=.25\textwidth]{age-define.eps}
%\caption{Evolution of the age of information over slots.}
%\label{fig:age-define}
%\end{figure}
Since the BS can update at most one user for each slot, $X_i(t) \geq 1$ for all $i$, $X_i(t)\neq X_j(t)$ for all $i \neq j$, and $\sum_{i=1}^N X_i(t) \geq 1+2+\cdots+N$ for all $t$.

\subsection{Problem formulation} \label{subsection:mdp}

%An algorithm is \textit{history dependent} if $D(t)$ depends on $D(0), \cdots, D(t-1)$ and $\mathbf{s}(0) \cdots, \mathbf{s}(t)$.   An algorithm is \textit{stationary} if $D(t_1)=D(t_2)$ when $\mathbf{s}(t_1)=\mathbf{s}(t_2)$ for any $t_1, t_2$.   Moreover, a \textit{randomized} algorithm specifies a probability distribution on the set of actions for each decision epoch, while a \textit{deterministic} algorithm chooses an action with certainty. 
We  define the \textit{average age} under  a scheduling algorithm $\theta$  by
\begin{align*}
\limsup_{T \rightarrow \infty} \frac{1}{T+1} E_{\theta} \left[ \sum_{t=0}^T \sum_{i=1}^N X_i(t)\right], 
\end{align*}
where $E_{\theta}$ represents the conditional expectation, given that the algorithm $\theta$ is employed. Note that we focus on the total age, but our work can be easily extended to a weighted sum of the ages.  
%\begin{define}
%A scheduling algorithm $\theta$ (that can be history dependent) is \textit{age-optimal} if it minimizes the  average  age.
%\end{define}
Our goal is to develop a low-complexity  scheduling algorithm whose average age is close to the minimum by leveraging the Whittle's methodology \cite{whittle}. 

\section{Scheduling algorithm design}
We will develop a scheduling algorithm based on restless bandits \cite{gittins2011multi} in stochastic control theory. To reach the goal, in this section, we start with casting our problem as a \textit{restless multi-armed bandit} problem \cite{gittins2011multi}, followed by introducing the Whittle index  \cite{whittle} as a solution to the multi-armed bandit problem. A challenge of this approach is to obtain the Whittle index. We then explicitly derive the Whittle index in a simpler way using a \textit{post-action age}. We finally propose a scheduling algorithm based on the Whittle index. 

\subsection{Restless bandits and Whittle's approach}

A restless bandit generalizes a classic bandit  by allowing the bandit to keep evolving under a \textit{passive} action, but in a distinct way from its  continuation under an \textit{active} action. However, the  restless  bandits problem, in general, is PSPACE-hard \cite{gittins2011multi}. Whittle hence investigates a relaxed version, where a constraint on the number of active bandits for each slot is replaced by the expected number.  With this relaxation, Whittle then applies a Lagrangian approach to decouple the multi-armed bandit  problem into multiple sub-problems. 

We can regard each user in our problem as a restless bandit. Following the Whittle's approach, we can decouple our problem into $N$ sub-problems. A sub-problem consists of a user $u_i$ and adheres to the network model in Section \ref{section:system} with $N$ = 1, except for an additional cost $C$ for updating the user. In each sub-problem, we aim at  determining whether or not the user should be  updated for each slot, in order to strike a balance between the updating cost and the cost incurred by age. In fact, the cost $C$ is a scalar Lagrange multiplier in the Lagrangian approach. Since each sub-problem consists of  a single user, hereafter we omit the index $i$ for simplicity.

\subsection{Decoupled sub-problem}
We  formulate the sub-problem as a Markov decision process (MDP), with the components \cite{MDP:Puterman} as follows.

\textbf{States}: We define the state $\mathbf{s}(t)$ of the MDP in slot $t$ by $\mathbf{s}(t)=(X(t), \Lambda(t))$. This is an infinite-state MDP as the age is possibly unbounded.    

\textbf{Actions}: Let $a(t) \in \{0, 1\}$ be an action of the MDP in slot $t$ indicating the BS's decision, where $a(t)=1$ if the BS decides \textit{to update} the user and $a(t)=0$ if the BS decides \textit{to idle}.

\textbf{Transition probabilities}:  
The transition probability from state $\mathbf{s}=(x, \lambda)$ to state $\mathbf{s}'$ under  action $a(t)=a$ is: 
\begin{align*}
&P[\mathbf{s}'=(x+1,1)|\mathbf{s}=(x, \lambda), a(t)=0]=p;\\
&P[\mathbf{s}'=(x+1,0)|\mathbf{s}=(x, \lambda), a(t)=0]=1-p;\\
&P[\mathbf{s}'=(1,1)|\mathbf{s}=(x, 1), a(t)=1]=p;\\
&P[\mathbf{s}'=(1,0)|\mathbf{s}=(x, 1), a(t)=1]=1-p;\\
&P[\mathbf{s}'=(x+1,1)|\mathbf{s}=(x, 0), a(t)=1]=p;\\
&P[\mathbf{s}'=(x+1,0)|\mathbf{s}=(x, 0), a(t)=1]=1-p.
\end{align*}

\textbf{Cost}: Let $C(\mathbf{s}(t), a(t))$ be an \textit{immediate cost} if action  $a(t)$ is taken in slot $t$ under  state $\mathbf{s}(t)$, with the definition as follows.
\begin{align}
C\Bigl(\textbf{s}(t)=(x, \lambda), a(t)=a\Bigr) \triangleq (x+1- x \cdot a\cdot \lambda) +C\cdot a, \label{eq:cost}
\end{align}
where the first part $x+1- x \cdot a\cdot \lambda$ is the resulting age in the next slot and the second part is the incurred cost for updating the user. 

A \textit{policy} $\mu=\{a(0), a(1), \cdots\}$ of the MDP  specifies an action for each slot. A policy $\mu$  is \textit{history dependent} if $a(t)$ depends on $\mathbf{s}(0), \cdots, \mathbf{s}(t)$ and $a(0) \cdots, a(t-1)$.  A policy is \textit{stationary} if $a(t_1)=a(t_2)$ when $\mathbf{s}(t_1)=\mathbf{s}(t_2)$ for any $t_1, t_2$.   Moreover, a \textit{randomized} policy chooses an action with a probability, while a \textit{deterministic} policy chooses an action with certainty. 

The \textit{average cost} under a policy $\mu$ is defined by
\begin{align*}
\limsup_{T \rightarrow \infty} \frac{1}{T+1} E_{\mu} \left[ \sum_{t=0}^T C(\mathbf{s}(t), a(t))\right]. 
\end{align*}

\begin{define}
A policy $\mu$ (that can be history dependent) is \textit{cost-optimal} if it minimizes the  average cost.
\end{define}
The objective of the MDP is to find a policy $\mu$ that minimizes the  average cost. According to \cite{MDP:Puterman}, there may not exist a cost-optimal policy that is stationary or deterministic. Hence, in the next section, we aim at investigating  a cost-optimal policy.

\subsection{Characterizing a cost-optimal policy}
We will study  structures of a cost-optimal policy  in this section. First, we show that a cost-optimal policy is stationary and deterministic as follows.

\begin{theorem} \label{theorem:stationary}
\begin{figure}[!t]
\centering
\includegraphics[width=.35\textwidth]{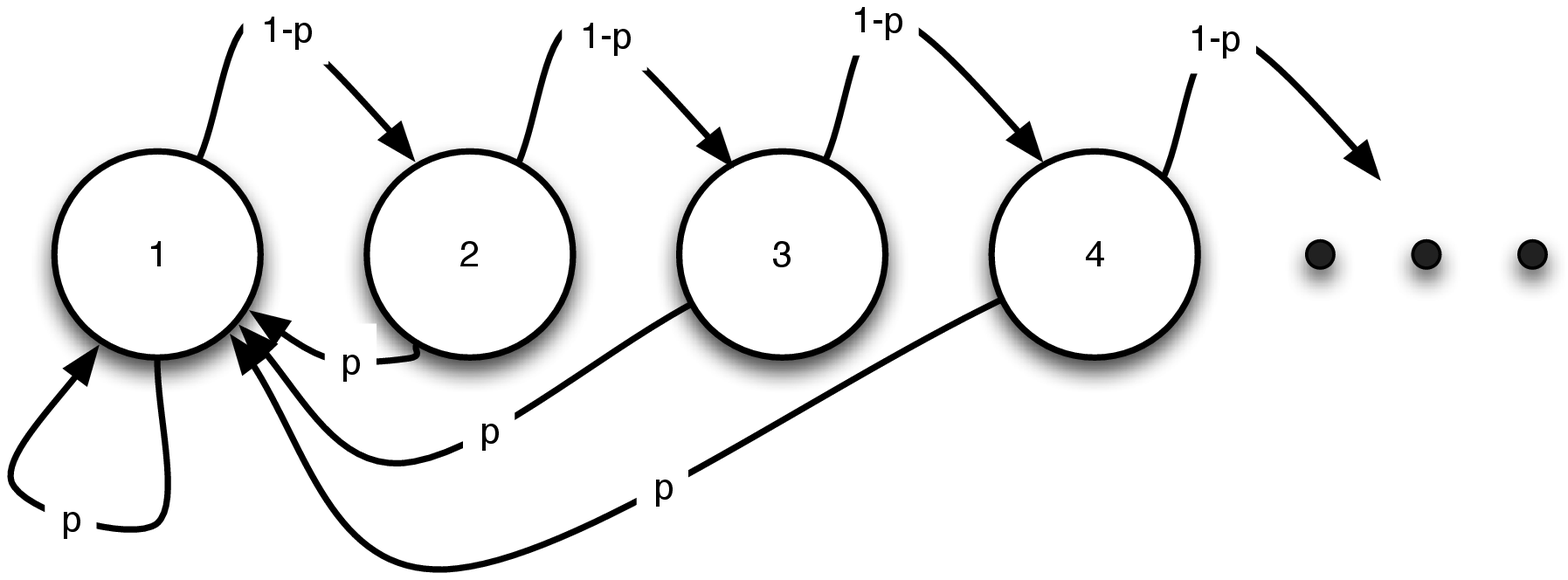}
\caption{The age $X(t)$ under the policy $f$ forms a DTMC.}
\label{fig:age-dtmc}
\end{figure}
There exists a stationary and deterministic policy that is cost-optimal, independent of  initial state $\mathbf{s}(0)$.
\end{theorem}
\begin{proof}
Given  initial state $\mathbf{s}(0)=\mathbf{s}$, we define the \textit{expected total  discounted cost} \cite{MDP:Puterman} under  policy $\mu$ by
\begin{align*}
J_{\alpha}(\mathbf{s};\mu)=\limsup_{T \rightarrow \infty} E_{\mu}\left[ \sum_{t=0}^T \alpha^t C(\mathbf{s}(t), a(t))|\mathbf{s}(0)=\mathbf{s} \right],
\end{align*}
where $0<\alpha<1$ is a discount factor. Moreover, let $J_{\alpha}(\mathbf{s})=\min_{\mu}J_{\alpha}(\mathbf{s};\mu)$ be the minimum   expected total discounted cost. A policy that minimizes the  expected total $\alpha$-discounted cost is called \textit{$\alpha$-optimal policy}.

According to \cite{stationary-policy:Sennott},  a deterministic stationary policy  is cost-optimal  if the following two conditions hold. 
\begin{enumerate}
%	\item \textit{$J_{\alpha}(\mathbf{s})$ is finite for all possible $\mathbf{s}$ and $\alpha$}: Suppose the initial age is $A(0)=A$. Let  $f$ be the policy of  always choosing the action $X(t)=0$ for each slot $t$. By the optimality of $J_{\alpha}(\mathbf{s})$, we obtain
%\begin{align*}
%J_{\alpha}(\mathbf{s})\leq& J_{\alpha}(\mathbf{s};f)\\
%=& \limsup_{T \rightarrow \infty}E_{f}\Bigl[\sum_{t=0}^{T} \alpha^t c(\mathbf{s}(t), D(t)) | \mathbf{s}(0)=\mathbf{s}\Bigr] \\
%=&\sum_{t=0}^{\infty} \alpha^t (A+t)=\frac{A}{1-\alpha}+ \frac{\alpha}{(1-\alpha)^2} < \infty.
%\end{align*}
%	

	\item  \textit{There exists a  deterministic stationary  policy of the MDP such that the associated  average cost is finite}: Let $f$ be the deterministic stationary policy of always choosing the action $a(t)=1$ for each slot $t$ if there is an arrival. 	The age $X(t)$ under the policy $f$ forms a discrete-time Markov chain (DTMC) in Fig. \ref{fig:age-dtmc}. The steady-state distribution $\boldsymbol{\pi}=(\pi_1, \pi_2, \cdots, )$ of the DTMC is 
\begin{align*}
	\pi_i=p(1-p)^{i-1}\,\,\,\,\text{for all $i=1,2, \cdots$}.
\end{align*}	
Hence, the average age is 
\begin{align*}
\sum_{i=1}^{\infty} i \pi_i= \sum_{i=1}^{\infty} i p(1-p)^{i-1} = \frac{1}{p}.
\end{align*}
On the other hand, the average updating cost is $C \cdot p$ as the arrival probability is $p$. Hence,  the  average cost under the policy $f$ is   the average age (i.e., $1/p$) plus the average updating cost (i.e., $C\cdot p$), which is finite and yields the result. 	
\item \textit{There exists a non-negative $L$ such that the relative cost function $h_{\alpha}(\mathbf{s}) \triangleq J_{\alpha}(\mathbf{s})-J_{\alpha}(\mathbf{0}) \geq -L$ for all $\mathbf{s}$ and $\alpha$, where $\mathbf{0}$ is a reference state}: Similar to  \cite{hsuage}, we can show that $J_{\alpha}(x,\lambda)$ is a non-decreasing function in  age $x$ given arrival indicator $\lambda$; moreover, $J_{\alpha}(x,\lambda)$ is a non-increasing function in $\lambda$ given $x$. Then, we can choose $L=0$ by  choosing the reference state  $\mathbf{0}=(0,1)$. 
\end{enumerate}
By verifying the two conditions, the theorem immediately follows from \cite{stationary-policy:Sennott}.
\end{proof}

Next, we further investigate a cost-optimal policy by showing 
 that it is a special type of  deterministic stationary  policy.

\begin{define}
A \textit{threshold-type} policy is a deterministic stationary   policy of the MDP. The action for state $(x,0)$ is to idle, for all $x$. 
Moreover, if the  action for state $(x,1)$ is to update, then the action for  state $(x+1,1)$ is to update as well. In other words, there exists a threshold $\bar{X} \in \{1,2, \cdots\}$ such that the action is to update if there is an arrival and the age is greater than or equal to $\bar{X}$; otherwise, the action is  to idle. 
\end{define}

\begin{theorem} \label{theorem:threshold}
If $C\geq 0$, then there exists a policy of the threshold type that is cost-optimal. 
\end{theorem}
\begin{proof}
It is obvious that an optimal action for state $(x,0)$ is to idle if $C \geq 0$. To establish the optimality of the threshold structure for state $(x,1)$, we need  the  \textit{discounted cost  optimality equation} (see the proof of Theorem \ref{theorem:stationary} and \cite{MDP:Puterman}): if the first condition in the proof of Theorem \ref{theorem:stationary} holds, then for any state $\mathbf{s}$ the minimum expected total discounted cost $J_{\alpha}(\mathbf{s})$ satisfies 
\begin{align*}
J_{\alpha}(\mathbf{s})=\min_{a \in \{0,1\}} C(\mathbf{s}, a)+\alpha E[J_{\alpha}(\mathbf{s}')],
\end{align*}
where the expectation is taken over all possible next state $\mathbf{s}'$ reachable from  state $\mathbf{s}$. 

Subsequently, we intend to prove that an $\alpha$-optimal policy is the threshold type. Let $\mathscr{J}_{\alpha}(\mathbf{s};a)=C(\mathbf{s}; a)+\alpha E[J_{\alpha}(\mathbf{s}')]$. Then, $J_{\alpha}(\mathbf{s})=\min_{a \in \{0,1\}} \mathscr{J}_{\alpha}(\mathbf{s};a)$. Moreover, an $\alpha$-optimal action for state $\mathbf{s}$ is $\argmin_{a \in \{0,1\}} \mathscr{J}_{\alpha}(\mathbf{s};a)$ \cite{MDP:Puterman}. Suppose an $\alpha$-optimal action for state $(x,1)$ is to update, i.e., 
\begin{align*}
\mathscr{J}_{\alpha}(x,1;1)-\mathscr{J}_{\alpha}(x,1;0) \leq 0.
\end{align*}
Then, an $\alpha$-optimal action for state $(x+1,1)$ is still to update since
\begin{align*}
&\mathscr{J}_{\alpha}(x+1,1;1)-\mathscr{J}_{\alpha}(x+1,1;0)\\
=&\left(1+C+\alpha E[J_{\alpha}(1,\lambda')] \right) -\left(x+2+\alpha E[J_{\alpha}(x+2,\lambda')] \right)\\
\mathop{\leq}^{(a)} & \left(1+C+\alpha E[J_{\alpha}(1,\lambda')] \right) -\left(x+1+\alpha E[J_{\alpha}(x+1,\lambda')] \right)\\
=&J_{\alpha}(x,1;1)-J_{\alpha}(x,1;0) \leq 0,
\end{align*}
where (a) results from the non-decreasing function of $J_{\alpha}(x,\lambda)$ in $x$ given $\lambda$ (see proof of Theorem \ref{theorem:stationary}). Hence, an $\alpha$-optimal policy is the threshold type.  

Finally, we conclude that a cost-optimal policy is the threshold type as well since a cost-optimal policy is the limit point of $\alpha$-optimal policies with $\alpha \rightarrow 1$ if both conditions in the proof of Theorem \ref{theorem:stationary} hold \cite{hsuage} .
\end{proof}

To find an optimal threshold for minimizing the average cost, we explicitly derive the average cost in the next theorem. 
\begin{figure}[!t]
\centering
\includegraphics[width=.35\textwidth]{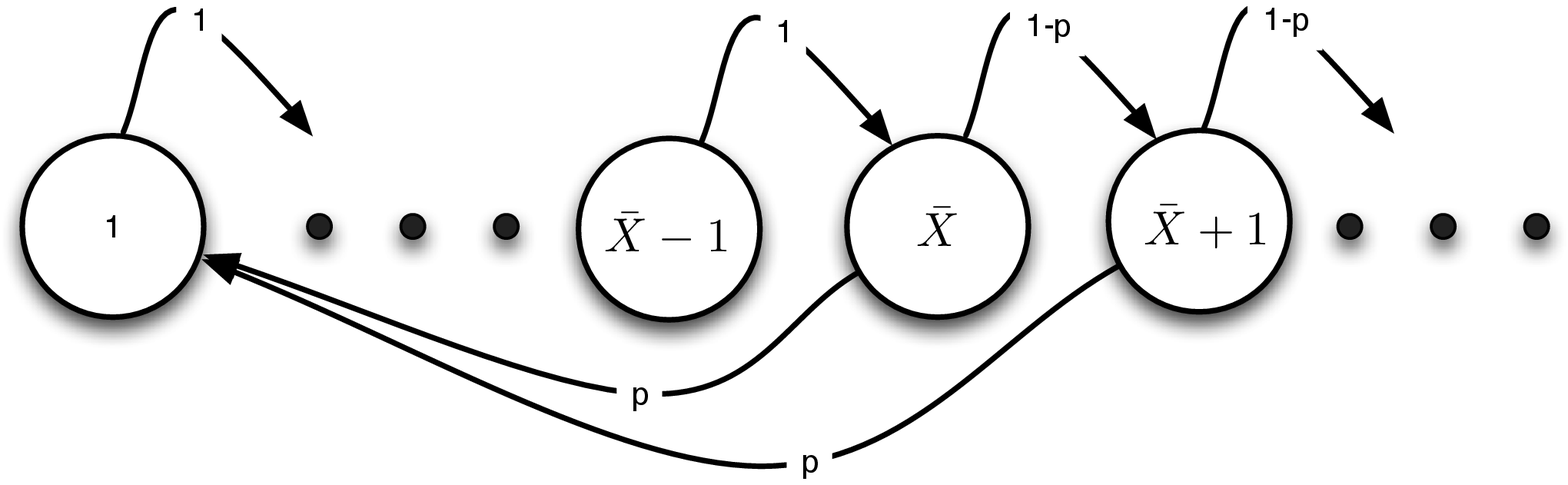}
\caption{The post-action age $Y(t)$ under the threshold-type policy forms a DTMC.}
\label{fig:post-age}
\end{figure}
\begin{theorem}
Given the threshold-type policy with the threshold $\bar{X}\in \{1, 2, \cdots\}$, then 
 the average  cost, denoted by $\mathscr{C}(\bar{X})$, under the policy is 
\begin{align}
\mathscr{C}(\bar{X})=\frac{\frac{\bar{X}^2}{2}+(\frac{1}{p}-\frac{1}{2})\bar{X}+\frac{1}{p^2}-\frac{1}{p}+C}{\bar{X}+\frac{1-p}{p}}. \label{eq:threshold-cost}
\end{align}
\end{theorem}
\begin{proof}
Let $Y(t)$ be the  age \textit{after} an action in slot $t$; precisely, if $\mathbf{s}(t)=(x,\lambda)$ and $a(t)=a$, then $Y(t)=x+1-x\cdot a\cdot \lambda$. Note that $Y(t)$, called  \textit{post-action age} (similar to the post-decsion state \cite{AMDP:Powell}), is  different from the \textit{pre-action age} $X(t)$. 

The post-action age $Y(t)$ forms a DTMC in Fig. \ref{fig:post-age}. 
Moreover, we associate each state in the DTMC with a cost. The DTMC incurs the cost of $C+1$ in slot $t$ when the post-action age in slot $t$ is $Y(t)=1$ since the post-action age $Y(t)=1$ implies that the BS updates the user, while incurring the cost of $y$ in slot $t$ when the post-action age is $Y(t)=y \neq 1$. Then, the steady-state distribution $\boldsymbol{\pi}=(\pi_1, \pi_2, \cdots)$ of the DTMC is  
\begin{align*}
\pi_i=\left\{
\begin{array}{ll}
\frac{1}{\bar{X}+\frac{1-p}{p}} & \text{if $i=1, \cdots, \bar{X}$;}\\
\frac{1}{\bar{X}+\frac{1-p}{p}}(1-p)^{i-\bar{X}} & \text{if $i=\bar{X}+1, \cdots$}.
\end{array}
\right.
\end{align*}

Therefore, the average cost of the DTMC is 
%\begin{align*}
%&(1+C)\pi_1 +\sum_{i=2}^{\infty} i\pi_i\\
%=&(1+C)\pi_1  +  \sum^{\bar{X}}_{i=2} i \pi_1 + \sum_{i=\bar{X}+1}^{\infty} i (1-p)^{i-\bar{X}}\pi_1 \\
%=&\frac{C+\frac{\bar{X}^2}{2}+(\frac{1}{p}-\frac{1}{2})\bar{X}+\frac{1}{p^2}-\frac{1}{p}}{\bar{X}+\frac{1-p}{p}}.
%\end{align*}
\begin{align*}
(1+C)\pi_1 +\sum_{i=2}^{\infty} i\pi_i=\frac{\frac{\bar{X}^2}{2}+(\frac{1}{p}-\frac{1}{2})\bar{X}+\frac{1}{p^2}-\frac{1}{p}+C}{\bar{X}+\frac{1-p}{p}}.
\end{align*}
\end{proof}

We want to elaborate on the post-action age in the proof. There might be  alternatives for obtaining the average cost as follows:
\begin{itemize}
	\item As in many literature, e.g., \cite{gittins2011multi}, we can find an optimal action for each state by solving the \textit{average cost optimality equation} \cite{MDP:Puterman}. However, we might not arrive at the average cost optimality equation as this is an infinite-state MDP \cite{stationary-policy:Sennott}. Even though  the average cost optimality equation is established, it is usually hard to solve the optimality equation directly. 
	\item Given threshold $\bar{X}$, then state $(X(t), \Lambda(t))$ forms a two-dimensional DTMC. It is usually hard to solve the steady-state distribution for a multi-dimensional DTMC. 
	\item Given  threshold $\bar{X}$, the pre-action age $X(t)$ forms a DTMC as well. However, we cannot associate each state with a cost, since the cost depends on not only state but also action (see Eq. (\ref{eq:cost})). On the contrary, the cost for the post-action age $Y(t)$ is determined by  state only. 
\end{itemize}

%\begin{theorem}
%An optimal threshold $\bar{X}$ of the threshold-type policy is
%
%\end{theorem}
%\begin{proof}
%
%\end{proof}
\subsection{Deriving the Whittle index}
Now, we are ready to define the Whittle index.  
\begin{define}
We define the Whittle index $I(\mathbf{s})$ by  the  cost $C$ that makes  both actions for state  $\mathbf{s}$ equally desirable. 
\end{define}

\begin{theorem} \label{theorem:whittle}
The Whittle index of the  sub-problem for state $(x,\lambda)$ is 
\begin{align*}
I(x,\lambda)=\left\{
\begin{array}{ll}
0 & \text{if $\lambda=0$;}\\
\frac{x^2}{2}-\frac{x}{2}+\frac{x}{p} & \text{if $\lambda=1$.}
\end{array}
\right.
\end{align*}
%Moreover, this is the  unique cost to make both actions for the state equally desirable. 
\end{theorem}
\begin{proof}
%It is obvious that the Whittle index for state $(x,0)$ is $I(x,0)=0$ as  both actions result in the same immediate cost if $C=0$. Regarding state $(x,1)$, we note that the age $x$ is \textit{around} the threshold  if   both actions for the state are equally desirable, since a cost-optimal policy is the threshold type. Notice that $\mathscr{C}(\bar{X})$ is a convex function in threshold $\bar{X}$, and hence the Whittle index $I(x,1)$ is the cost $C$ satisfies
It is obvious that the Whittle index for state $(x,0)$ is $I(x,0)=0$ as  both actions result in the same immediate cost and the same age of next slot if $C=0$. 

Let $g(x)=\mathscr{C}(x)$ in Eq. (\ref{eq:threshold-cost})  for the domain of $\{x\in \mathbb{R}: x\geq 1\}$. Note that $g(x)$ is strictly convex in the domain. Let $x^*$ be the minimizer of $g(x)$. Then, an optimal threshold for minimizing the average cost $\mathscr{C}(\bar{X})$ is either $\lfloor x^* \rfloor$ or $\lceil x^* \rceil$: the optimal threshold is $\bar{X}^*=\lfloor x^* \rfloor$ if $\mathscr{C}(\lfloor x^* \rfloor) \leq  \mathscr{C}(\lceil x^* \rceil)$ and $\bar{X}^*=\lceil x^* \rceil$  if $\mathscr{C}(\lceil x^* \rceil) \leq \mathscr{C}(\lfloor x^* \rfloor)$. If there is a tie, both choices are optimal.  

Hence,  both actions for state $(x,1)$ are equally desirable if and only if the age $x$ satisfies
\begin{align*}
\mathscr{C}(x)=\mathscr{C}(x+1),
\end{align*}
i.e., $x= \lfloor x^* \rfloor$ and both thresholds of $x$ and $x+1$ are optimal. 
By solving the above equation, we  obtain the cost to make both actions equally desirable, as the Whittle index in the theorem. 
\end{proof}

According to Theorem \ref{theorem:whittle}, both actions might have a tie. If there is a tie, we break the tie in favor of idling. Then, we can explicitly express the optimal threshold  in the next theorem. 

\begin{theorem} \label{theorem:optimal-threshold}
The optimal threshold for minimizing the average cost is $\bar{X}^*=x$ if the cost satisfies $I(x-1,1) \leq  C < I(x,1)$, for all $x=1, 2, \cdots$.
\end{theorem}
\begin{proof}
Since $I(x)$ is the  cost to make both actions for state $(x,1)$ equally desirable and we break a tie in favor of idling, then the optimal threshold is $\bar{X}^*=x+1$ if the cost is $C=I(x)$, for all $x$. We  claim that  the optimal threshold monotonically increases with cost $C$, and then the theorem follows. 

To verify the claim,  we can focus on an $\alpha$-optimal policy according to the proof of Theorem \ref{theorem:threshold}.  Suppose that an $\alpha$-optimal action, associated with a cost $C_1$, for state $(x,1)$ is to idle, i.e., 
\begin{align*}
x+1+\alpha E[J_{\alpha}(x+1,\lambda')] \leq 1+C_1+\alpha E[J_{\alpha}(1,\lambda')] .
\end{align*}
Then, an $\alpha$-optimal action, associated with a cost $C_2 \geq C_1$, for state $(x,1)$ is to idle as well since 
\begin{align*}
x+1+\alpha E[J_{\alpha}(x+1,\lambda')] \leq &1+C_1+\alpha E[J_{\alpha}(1,\lambda')] \\
\leq & 1+C_2+\alpha E[J_{\alpha}(1,\lambda')].
\end{align*}
Then, the monotonicity is established.  
\end{proof}

Next, according to \cite{whittle}, we have to demonstrate the \textit{indexability}  defined as follows.  
\begin{define}
Given cost $C$, let $\mathbf{S}(C)$ be the set of states $\mathbf{s}$ such that the optimal action for the states is to idle. The sub-problem is \textit{indexable} if  the set $\mathbf{S}(C)$ monotonically increases from the empty set to the entire state space, as $C$ increases from $-\infty$ to $\infty$. 
\end{define}

\begin{theorem}
The sub-problem is indexable. 
\end{theorem}
\begin{proof}
If $C<0$, the optimal action for every state is to update; as such, $\mathbf{S}(0)=\emptyset$. If $C\geq 0$, then $\mathbf{S}(C)$ is composed of the set $\{\mathbf{s}=(x,0): x=1, 2, \cdots\}$  and a set of $(x,1)$ for some $x$'s. According to Theorem \ref{theorem:optimal-threshold}, the optimal threshold monotonically increases as $C$ increases, and hence the set $\mathbf{S}(C)$ monotonically increases to the entire state space. 
\end{proof}

\subsection{Scheduling algorithm design}

Now, we are ready to propose a scheduling algorithm based on the Whittle index.   For each slot $t$, the BS observes  age $X_i(t)$ and arrival indicator $\Lambda_i(t)$ for every user $u_i$; then, update a user $u_i$ with the highest value of the Whittle index $I(X_i(t),\Lambda_i(t))$. We can think of the index $I(X_i(t),\Lambda_i(t))$ as the cost to update  user $u_i$. The intuition of the scheduling algorithm is that the BS intends to send the most valuable packet. In  Fig. \ref{fig:whittle}, we compare the proposed algorithm with the age-optimal scheduling algorithm in \cite{hsuage} for two users over 100,000 slots. It turns out that the simple index algorithm almost achieves the minimum average age.

%\textbf{On-line scheduling algorithm}: Due to the absence of the arrival probabilities, for each slot $t$ the BS estimates  the arrival probability of user $u_i$ by $\bar{p}_i(t)=\sum_{\tau=0}^{t}(\Lambda_i(\tau))/(t+1)$, and then update a user $u_i$ with the highest value of the Whittle index $I(X_i(t),\Lambda_i(t))$ with $p_i$ being substituted by $\bar{p}_i(t)$. 

%\begin{minipage}{.5\textwidth}
\begin{figure}
\centering
\includegraphics[width=.4\textwidth]{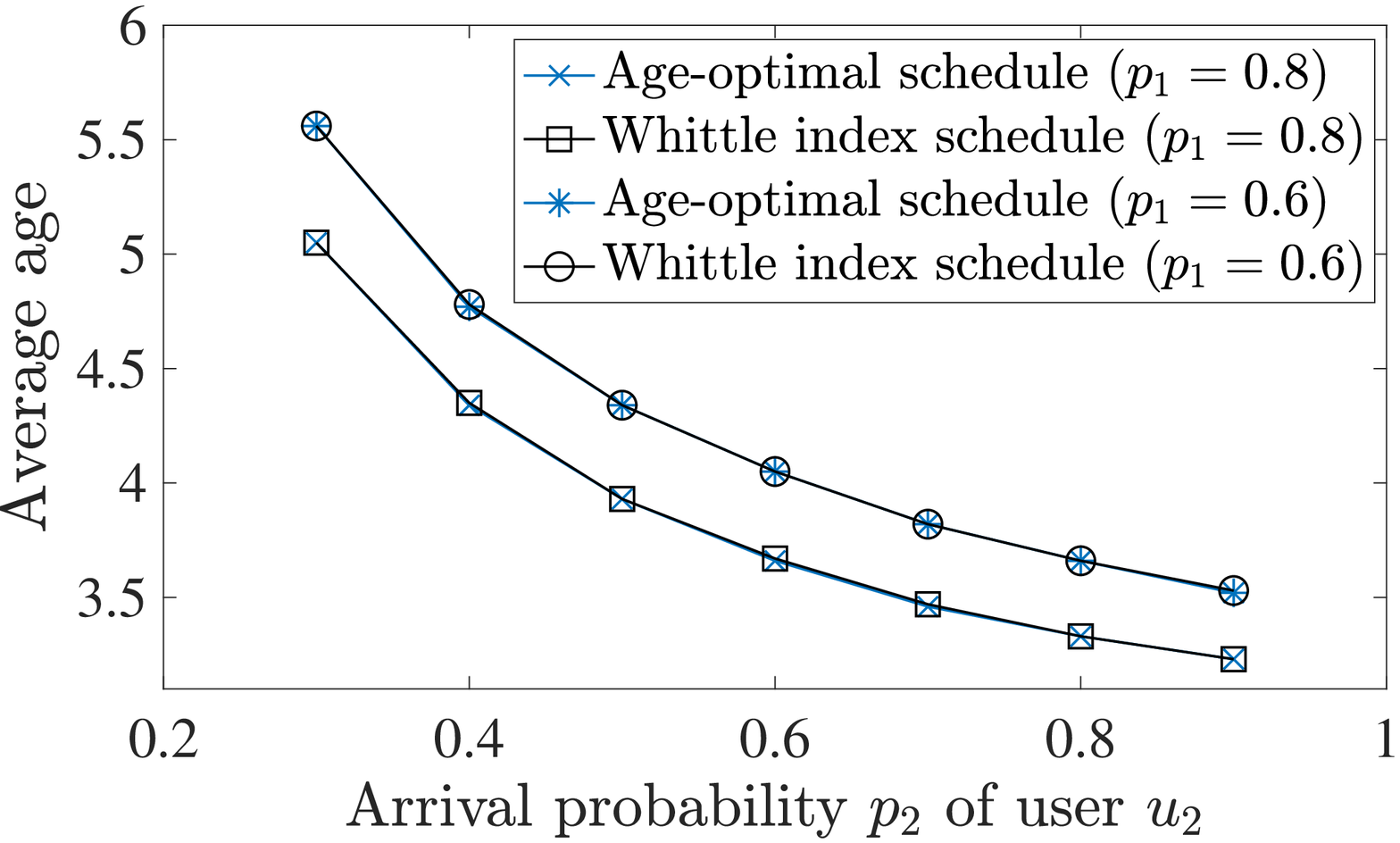}
\caption{Average age for two users under the proposed Whittle index scheduling algorithm and an age-optimal scheduling algorithm.}
\label{fig:whittle}
\end{figure}

\section{Conclusion}
This paper treated a wireless broadcast network, where many users are interested in different types of information delivered by a base-station. Under a transmission constraint, we studied a transmission scheduling problem, with respect to the age of information. We have proposed a low-complexity scheduling algorithm  leveraging the Whittle's methodology. Numerical studies showed that the proposed algorithm almost minimizes the average age. To investigate a regime under which the proposed algorithm is  optimal would be an interesting extension.  

\section*{Acknowledgement}
The author is grateful to anonymous reviewers for their constructive comments. 

{\small
\bibliographystyle{IEEEtran}
\bibliography{IEEEabrv,ref}

\begin{thebibliography}{10}
\providecommand{\url}[1]{#1}
\csname url@samestyle\endcsname
\providecommand{\newblock}{\relax}
\providecommand{\bibinfo}[2]{#2}
\providecommand{\BIBentrySTDinterwordspacing}{\spaceskip=0pt\relax}
\providecommand{\BIBentryALTinterwordstretchfactor}{4}
\providecommand{\BIBentryALTinterwordspacing}{\spaceskip=\fontdimen2\font plus
\BIBentryALTinterwordstretchfactor\fontdimen3\font minus
  \fontdimen4\font\relax}
\providecommand{\BIBforeignlanguage}[2]{{%
\expandafter\ifx\csname l@#1\endcsname\relax
\typeout{** WARNING: IEEEtran.bst: No hyphenation pattern has been}%
\typeout{** loaded for the language `#1'. Using the pattern for}%
\typeout{** the default language instead.}%
\else
\language=\csname l@#1\endcsname
\fi
#2}}
\providecommand{\BIBdecl}{\relax}
\BIBdecl

\bibitem{age:kaul}
S.~Kaul, R.~D. Yates, and M.~Gruteser, ``{Real-Time Status: How Often Should
  One Update?}'' \emph{Proc of IEEE INFOCOM}, pp. 2731--2735, 2012.

\bibitem{hsuage}
Y.-P. Hsu, E.~Modiano, and L.~Duan, ``{Age of Information: Design and Analysis
  of Optimal Scheduling Algorithms},'' \emph{Proc. of IEEE ISIT}, pp. 561--565,
  2017.

\bibitem{gittins2011multi}
J.~Gittins, K.~Glazebrook, and R.~Weber, \emph{{Multi-Armed Bandit Allocation
  Indices}}.\hskip 1em plus 0.5em minus 0.4em\relax John Wiley \& Sons, 2011.

\bibitem{whittle}
P.~Whittle, ``{Restless Bandits: Activity Allocation in a Changing World},''
  \emph{Journal of applied probability}, pp. 287--298, 1988.

\bibitem{weber1990index}
R.~R. Weber and G.~Weiss, ``{On an index policy for restless bandits},''
  \emph{Journal of Applied Probability}, vol.~27, no.~3, pp. 637--648, 1990.

\bibitem{larranaga2015stochastic}
M.~Larranaga, U.~Ayesta, and I.~M. Verloop, ``{Stochastic and Fluid Index
  Policies for Resource Allocation Problems},'' \emph{Proc. of IEEE INFOCOM},
  pp. 1230--1238, 2015.

\bibitem{MDP:Bertsekas}
D.~P. Bertsekas, \emph{{Dynamic Programming and Optimal Control Vol. I and
  II}}.\hskip 1em plus 0.5em minus 0.4em\relax Athena Scientific, 2012.

\bibitem{age:Costa}
M.~Costa, M.~Codreanu, and A.~Ephremides, ``{On the Age of Information in
  Status Update Systems with Packet Management},'' \emph{{IEEE} Trans. Inf.
  Theory}, vol.~62, no.~4, pp. 1897--1910, 2016.

\bibitem{bedewy2017age}
A.~M. Bedewy, Y.~Sun, and N.~B. Shroff, ``{Age-Optimal Information Updates in
  Multihop Networks},'' \emph{Proc. of IEEE ISIT}, pp. 576--580, 2017.

\bibitem{kosta2017age}
A.~Kosta, N.~Pappas, and V.~Angelakis, ``{Age of Information: A New Concept,
  Metric, and Tool},'' \emph{Foundations and Trends{\textregistered} in
  Networking}, vol.~12, no.~3, pp. 162--259, 2017.

\bibitem{age:he}
Q.~He, D.~Yuan, and A.~Ephremides, ``{Optimal Link Scheduling for Age
  Minimization in Wireless Systems},'' \emph{accepted to IEEE Trans. Inf.
  Theory}, 2017.

\bibitem{joowireless}
C.~Joo and A.~Eryilmaz, ``{Wireless Scheduling for Information Freshness and
  Synchrony: Drift-based Design and Heavy-Traffic Analysis},'' \emph{Proc. of
  IEEE WIOPT}, pp. 1--8, 2017.

\bibitem{age:igor}
I.~Kadota, E.~Uysal-Biyikoglu, R.~Singh, and E.~Modiano, ``{Minimizing Age of
  Information in Broadcast Wireless Networks},'' \emph{Proc. of Allerton},
  2016.

\bibitem{yatesage}
R.~D. Yates, P.~Ciblat, A.~Yener, and M.~Wigger, ``{Age-Optimal Constrained
  Cache Updating},'' \emph{Proc. of IEEE ISIT}, pp. 141--145, 2017.

\bibitem{MDP:Puterman}
M.~L. Puterman, \emph{{Markov Decision Processes: Discrete Stochastic Dynamic
  Programming}}.\hskip 1em plus 0.5em minus 0.4em\relax The MIT Press, 1994.

\bibitem{stationary-policy:Sennott}
L.~I. Sennott, ``{Average Cost Optimal Stationary Policies in Infinite State
  Markov Decision Processes with Unbounded Costs},'' \emph{Operations
  Research}, vol.~37, pp. 626--633, 1989.

\bibitem{AMDP:Powell}
W.~B. Powell, \emph{{Approximate Dynamic Programming: Solving the Curses of
  Dimensionality}}.\hskip 1em plus 0.5em minus 0.4em\relax John Wiley \& Sons,
  2011.

\end{thebibliography}
}

\end{document}